\documentclass[12pt]{amsart}
\usepackage{amsmath,amsfonts,amssymb,amsthm,amsxtra,amscd}
\usepackage{mathrsfs} 
\usepackage{graphicx} 
\theoremstyle{plain} 
\newtheorem{thm}{Theorem}[] 
\newtheorem{prop}[thm]{Proposition} 
\newtheorem{lemma}[thm]{Lemma}

\theoremstyle{definition}

\newtheorem{remark}[thm]{Remark}

\newcommand{\z}{{\mathbb Z}} 
\newcommand{\pj}{{{\mathbb P}^1}}

\newcommand{\piii}{{{\mathbb P}^3}}

\newcommand{\sce}{\mathscr{E}}
\newcommand{\scf}{\mathscr{F}} 
\newcommand{\scg}{\mathscr{G}}
\newcommand{\sco}{\mathscr{O}} 
\newcommand{\sch}{\mathscr{H}}
\newcommand{\sci}{\mathscr{I}}

\newcommand{\sct}{\mathscr{T}}

\newcommand{\tH}{\text{H}} 
\newcommand{\h}{\text{h}}

\newcommand{\izo}{\overset{\sim}{\rightarrow}} 
 
\newcommand{\ra}{\rightarrow} 
\newcommand{\lra}{\longrightarrow} 
\newcommand{\xra}{\xrightarrow}  
\newcommand{\vb}{\, \vert \, } 
\newcommand{\prim}{{\, \prime}} 
\newcommand{\secund}{{\prime \prime}}
\newcommand{\Ker}{\text{Ker}\, }
\newcommand{\Cok}{\text{Coker}\, }

\begin{document}

\title[4-instantons]{Four generated 4-instantons} 

\author[Anghel,~Coand\u{a}~and~Manolache]{Cristian~Anghel,~Iustin~Coand\u{a}~
and~Nicolae~Manolache}
\address{Institute of Mathematics of the Romanian Academy, P.O. Box 1--764, 
RO--014700, Bucharest, Romania}
\email{Iustin.Coanda@imar.ro~Cristian.Anghel@imar.ro~Nicolae.Manolache@imar.ro}



\begin{abstract}
We show that there exist mathematical 4-instanton bundles $F$ on the 
projective 3-space such that $F(2)$ is globally generated (by four global 
sections). This is equivalent to the existence of elliptic space curves of 
degree 8 defined by quartic equations. There is a (possibly incomplete) 
intersection theoretic argument for the existence of such curves in 
D'Almeida [Bull. Soc. Math. France 128 (2000), 577--584] and another argument, 
using results of Mori [Nagoya Math. J. 96 (1984), 127--132], in Chiodera and 
Ellia [Rend. Istit. Univ. Trieste 44 (2012), 413--422]. Our argument is 
quite different. We prove directly the former fact, using the method of 
Hartshorne and Hirschowitz [Ann. Scient. \'{E}c. Norm. Sup. (4) 15 (1982), 
365--390] and the geometry of five lines in the projective 3-space.    
\end{abstract}

\maketitle

A mathematical $n$-instanton bundle on $\piii$ ($n$-instanton, for short) is 
a rank 2 vector bundle $F$ on $\piii$, with $c_1(F) = 0$, $c_2(F) = n$, such 
that $\tH^i(F(-2)) = 0$, $i = 0, \ldots , 3$ (since $c_1(F) = 0$ one has 
$F \simeq F^\vee$ hence, by Serre duality, $\h^i(F(-2)) = \h^{3-i}(F(-2))$, 
$i = 0,\, 1$, which implies that $\chi(F(-2)) = 0$). Let us recall that if 
$Y$ is a closed subscheme of $\piii$ of dimension 1 then $\tH^1(\sco_Y(-2)) 
\neq 0$ (if $C$ is a reduced and irreducible closed subscheme of $Y$ of 
dimension 1 then the map $\tH^1(\sco_Y(-2)) \ra \tH^1(\sco_C(-2))$ is 
surjective and $\tH^1(\sco_C(-2)) \neq 0$ because $\chi(\sco_C(-2)) < 0$ by 
Riemann-Roch on $C$; this argument, which appears in the proof of 
\cite[Lemma~1]{po}, was suggested by C. B\u{a}nic\u{a}). It follows that if  
$Y$ is locally complete intersection with $\omega_Y \simeq \sco_Y(m)$ 
then $\h^0(\sco_Y(m+2)) = \h^0(\omega_Y(2)) = \h^1(\sco_Y(-2)) > 0$. One 
deduces easily that if $\tH^0(F(-1)) \neq 0$ then $F \simeq \sco_\piii(1) 
\oplus \sco_\piii(-1)$ (hence $n = -1$), and if $\tH^0(F(-1)) = 0$ and 
$\tH^0(F) \neq 0$ then $F \simeq 2\sco_\piii$ (hence $n = 0$). If $\tH^0(F) 
= 0$ then $F$ is stable hence $n \geq 1$. Examples of $n$-instantons are 
the bundles that can be obtained as extensions$\, :$ 
\begin{equation}\label{E:tHooft} 
0 \lra \sco_\piii(-1) \lra F \lra \sci_{L_1 \cup \ldots \cup L_{n+1}}(1) \lra 0
\end{equation}  
where $L_1, \ldots , L_{n+1}$ are mutually disjoint lines in $\piii$. For 
$n \leq 2$, all $n$-instantons can be obtained in this way. This is no longer 
true for $n \geq 3$. 

We are concerned with the problem of the global generation of twists of 
instantons. It is well known that if $F$ is an $n$-instanton then $F$ is 
$n$-regular hence $F(n)$ is globally generated (the argument is recalled in 
\cite[Remark~4.7]{acm}). Gruson and Skiti \cite{gs} showed that if $F$ is a 
3-instanton having no jumping line of maximal order 3 then $F(2)$ is globally 
generated. Our aim here is to prove the following$\, :$ 

\begin{prop}\label{P:f(2)gg} 
There exist $4$-instantons $F$ on $\piii$ such that $F(2)$ is globally 
generated. 
\end{prop} 

It is shown in \cite[Remark~6.4]{acm} that if $F$ is a 4-instanton with $F(2)$ 
globally generated then $\tH^0(F(1)) = 0$ and $\tH^1(F(2)) = 0$ (hence 
$\h^0(F(2)) = 4$). It follows that the 4-instantons $F$ with $F(2)$ globally 
generated form a nonempty open subset of the moduli space of 4-instantons.  

To see that the result stated above is nontrivial we recall that, by a 
theorem of Rahavandrainy \cite{ra}, a general 4-instanton $F$ admits a minimal 
free resolution of the form$\, :$ 
\[
0 \lra 4\sco_\piii(-5) \lra 10\sco_\piii(-4) \lra 4\sco_\piii(-2) \oplus 
4\sco_\piii(-3) \lra F \lra 0\, . 
\] 

A proof of Prop.~\ref{P:f(2)gg} was given by Chiodera and Ellia \cite{ce}. 
They show that there exist elliptic curves $X$ in $\piii$ of degree 8 with 
$\sci_X(4)$ globally generated and construct $F$ as an extension$\, :$ 
\[
0 \lra \sco_\piii(-2) \lra F \lra \sci_X(2) \lra 0\, . 
\]
The proof of the existence of such elliptic curves uses results of Mori 
\cite{mo} (in particular, it uses the fact that there exist nonsingular 
quartic surfaces $\Sigma \subset \piii$ with $\text{Pic}\, \Sigma = \z H 
\oplus \z C$, where $H$ denotes a plane section and $C$ is an elliptic curve 
of degree 8). Another argument for the existence of this kind of elliptic 
curves appears in D'Almeida \cite{da} but it seems to be incomplete, according 
to Chiodera and Ellia \cite[Remark~2.11]{ce}. 

\vskip2mm 

We prove Prop.~\ref{P:f(2)gg} in a quite different, more elementary, way  
using the method of Hartshorne and Hirschowitz \cite{hh}. The key point 
of our proof is the following$\, :$ 

\begin{lemma}\label{L:omega(1)iy(3)} 
Let $L_1 , \ldots , L_5$ be mutually disjoint lines in $\piii$ such that 
their union admits no $5$-secant. Then there exist epimorphisms$\, :$ 
\[
\Omega_\piii(1) \lra \sci_{L_1 \cup \ldots \cup L_5}(3) \lra 0\, . 
\]
\end{lemma} 

Assuming this lemma, for the moment, let us give the 

\begin{proof}[Proof of Prop.~\ref{P:f(2)gg}] 
Let $L_1 , \ldots , L_5$ be five mutually disjoint lines in $\piii$ such that 
their union $Y$ admits no 5-secant. One considers, 
following \cite{hh}, the torsion free sheaf $\scf_0 = \sco_\piii(-1) \oplus 
\sci_Y(1)$. It is classically known that $\tH^1(\sci_Y(3)) = 0$ (we recall the 
argument, for completeness, in the proof of Lemma~\ref{L:hiiy(3)=0} below). 
It follows that $\scf_0$ is 3-regular. Moreover, $\h^0(\scf_0(3)) = 20$. 
$\scf_0$ corresponds to a point $q_0$ of 
Grothendieck's \emph{Quot scheme} $Q$ parametrizing the quotients a 
$20\sco_\piii(-3)$ having the same Hilbert polynomial as $\scf_0$. 
By \cite[\S 4]{hh}, $Q$ is nonsingular at $q_0$. Let $\sce$ be the universal 
family on $\piii \times Q$ (such that $\sce_{q_0} \simeq \scf_0$) and let 
$Q_0$ be the irreducible component of $Q$ containing $q_0$. 

Consider, now, an epimorphism $\sigma : \Omega_\piii(1) \ra \sci_Y(3)$ and let 
$\scg$ be the cokernel of the composite morphism$\, :$ 
\[
\Ker \sigma \lra \Omega_\piii(1) \lra 4\sco_\piii \, .  
\] 
Using the exact sequence $0 \ra \Omega_\piii(1) \ra 4\sco_\piii \ra \sco_\piii(1) 
\ra 0$, one sees easily that $\scg$ can be realized as an extension$\, :$ 
\[
0 \lra \sci_Y(3) \lra \scg \lra \sco_\piii(1) \lra 0\, . 
\]
$\scg$ is 1-regular (because $\sci_Y(3)$ is) and globally generated (it is a 
quotient of $4\sco_\piii$). Since, by the universal property of the Quot 
scheme, there exists a canonical morphism$\, :$ 
\[
\text{Ext}^1(\sco_\piii(-1), \sci_Y(1)) \lra Q_0 
\]  
(here $\text{Ext}^1$ is regarded as an affine space) it follows that there 
exist points $q \in Q_0$ such that $\sce_q(2)$ is $1$-regular and globally 
generated. These points form a dense open subset of $Q_0$. 

On the other hand, since there exists a canonical morphism$\, :$ 
\[
\text{Ext}^1(\sci_Y(1), \sco_\piii(-1)) \lra Q_0
\] 
it follows that there exist points $q \in Q_0$ such that $\sce_q$ is locally 
free, with $\tH^0(\sce_q) = 0$ and $\tH^1(\sce_q(-2)) = 0$. These points 
form an open subset of $Q_0$. 

If $q$ is a point in the intersection of these two open subsets of $Q_0$ then 
$\sce_q$ is a 4-instanton with $\sce_q(2)$ globally generated.  
\end{proof}  

The proof of Lemma~\ref{L:omega(1)iy(3)} will occupy the rest of the  
paper. We begin by recalling some easy lemmata. 

\begin{lemma}\label{L:qcapqprim} 
Let $L_1 , \ldots , L_4$ be mutually disjoint lines in $\piii$ such that their 
union is not contained in a quadric surface. Let $Q$ (resp., $Q^\prime$) be the 
quadric surface containing $L_1 \cup L_2 \cup L_3$ (resp., $L_2 \cup L_3 \cup 
L_4$). If $L_4 \cap Q$ consists of two points, put $X = L \cup L^\prime$, where 
$L$ and $L^\prime$ are the lines from the other ruling of $Q$ containing those 
points. If $L_4 \cap Q$ consists of only one point, denote by $X$ the divisor 
$2L$ on $Q$, where $L$ is the line from the other ruling of $Q$ containing 
that point. Then, as schemes$\, :$ 
\[
Q \cap Q^\prime = L_2 \cup L_3 \cup X\, . 
\]
\end{lemma} 

\begin{proof}
As divisors on $Q$, $Q \cap Q^\prime = L_2 + L_3 + \Lambda + \Lambda^\prime$, 
where $\Lambda$ and $\Lambda^\prime$ are (not necessarily distinct) lines from 
the other ruling of $Q$. Since $\Lambda$ and $\Lambda^\prime$ must intersect 
$L_4$ (because they intersect $L_2$ and $L_3$ and are contained in $Q^\prime$) 
the assertion follows. 
\end{proof} 

\begin{lemma}\label{L:il1l4x} 
Under the hypothesis of Lemma~\ref{L:qcapqprim}, the ideal sheaf of 
$L_1 \cup \ldots \cup L_4 \cup X$ has a resolution of the form$\, :$ 
\[
0 \lra 3\sco_\piii(-4) \lra 4\sco_\piii(-3) \lra 
\sci_{L_1 \cup \ldots \cup L_4 \cup X} \lra 0\, . 
\]
\end{lemma}

\begin{proof}
Let us denote $L_1 \cup \ldots \cup L_4 \cup X$ by $Z$. Let $\Lambda \subset 
Q$ (resp., $\Lambda^\prime \subset Q^\prime$) be a line from the same ruling 
of $Q$ (resp., $Q^\prime$) as the lines from the support of $X$ but not 
intersecting $X$. Let $H$ (resp., $H^\prim$) be the plane spanned by $L_1$ and 
$\Lambda$ (resp., $L_4$ and $\Lambda^\prime$). One has$\, :$ 
\[
(Q \cup H^\prim) \cap (Q^\prime \cup H) = Z \cup \Lambda \cup \Lambda^\prime 
\cup (H \cap H^\prim)\, . 
\] 
One can choose $\Lambda$ and $\Lambda^\prime$ such that they are, moreover, 
disjoint. Both of them intersect the line $H \cap H^\prim$. It follows that 
$Z^\prim := \Lambda \cup \Lambda^\prime \cup (H \cap H^\prim)$ is a divisor of 
type $(2,1)$ on a nonsingular quadric surface (take a line $\Lambda^\secund$ 
disjoint from $\Lambda$ and $\Lambda^\prime$ but intersecting $H \cap H^\prim$) 
hence it is linked to a line by a complete intersection of type $(2,2)$. 
One can apply, now, D. Ferrand's result on resolutions under liaison (see  
\cite[Prop.~2.5]{ps}). 
\end{proof}

\begin{lemma}\label{L:hiiy(3)=0} 
Let $L_1, \ldots , L_5$ be mutually disjoint lines in $\piii$ such that their 
union $Y$ admits no $5$-secant. Then ${\fam0 H}^i(\sci_Y(3)) = 0$, 
$i = 0,\, 1$. 
\end{lemma}

\begin{proof} 
Let $Q$ be the nonsingular quadric surface containing $L_1 \cup L_2 \cup L_3$. 
Since $Y$ has no 5-secant, none of the lines $L_4$ and $L_5$ is contained in 
$Q$ hence the scheme $\Gamma := Q \cap (L_4 \cup L_5)$ is 0-dimensional of 
length 4. 

Now, since $\h^0(\sco_\piii(3)) = 20 = \h^0(\sco_Y(3))$ it suffices to show 
that $\tH^0(\sci_Y(3)) = 0$. Assume that $Y$ is contained in a cubic surface 
$\Sigma$. $Q$ cannot be a component of $\Sigma$ because $L_4 \cup L_5$ is not 
contained in a plane. It follows that $\Sigma$ intersects $Q$ properly hence, 
as divisors on $Q$, 
\[
\Sigma \cap Q = L_1 + L_2 + L_3 + L_1^\prime + L_2^\prime + L_3^\prime 
\]   
where $L_1^\prime$, $L_2^\prime$, $L_3^\prime$ are (not necessarily distinct) lines 
from the other ruling of $Q$. One deduces that $\Gamma$ must be a subscheme 
of the (effective) divisor $L_1^\prime + L_2^\prime + L_3^\prime$. But $Y$ has no 
5-secant hence the scheme $L_i^\prime \cap \Gamma = L_i^\prime \cap 
(L_4 \cup L_5)$ is empty or consists of a simple point, $i = 1,\, 2,\, 3$. 
Since $\Gamma$ has length 4, one gets a \emph{contradiction}.    
\end{proof} 

\begin{lemma}\label{L:l5} 
Let $L_1, \ldots ,L_5$ be mutually disjoint lines in $\piii$. For $1 \leq i < 
j < l \leq 5$, let $Q_{ijl}$ be the unique quadric surface containing $L_i \cup 
L_j \cup L_l$. If $L_1 \cup \ldots \cup L_5$ admits no $5$-secant then, 
as schemes$\, :$ 
\[
Q_{125} \cap Q_{235} \cap Q_{345} = L_5 \cup \Gamma_2 \cup \Gamma_3 
\]
where $\Gamma_i$ is a subscheme of length $2$ of $L_i$, $i = 2,\, 3$. 
\end{lemma}

\begin{proof}
Since $L_1 \cup \ldots \cup L_5$ admits no 5-secant the union of any four of 
the five lines is not contained in a quadric surface. 
According to Lemma~\ref{L:qcapqprim} one has, as schemes$\, :$ 
\[
Q_{125} \cap Q_{235} = L_2 \cup L_5 \cup X\, ,\  \  
Q_{235} \cap Q_{345} = L_3 \cup L_5 \cup X^\prim 
\]
where $X$ (resp., $X^\prim$) is the divisorial sum on $Q_{125}$ (resp., 
$Q_{235}$) of the two (maybe coinciding) 4-secants of $L_1 \cup L_2 \cup L_3 
\cup L_5$ (resp., $L_2 \cup L_3 \cup L_4 \cup L_5$). Since $L_1 \cup \ldots 
\cup L_5$ admits no 5-secant one has $X \cap X^\prim = \emptyset$. 
Putting $\Gamma_2 := L_2 \cap X^\prim$ and $\Gamma_3 = L_3 \cap X$, the result 
follows.  
\end{proof} 

\begin{remark}\label{R:theta} 
We describe, here, explicitly some special morphisms 
$\Omega_\piii(1) \ra \sco_\piii(1)$ that we shall use in the proof of 
Lemma~\ref{L:omega(1)iy(3)} and record their basic properties. 

\vskip2mm 

(i) We denote by $W$ the space $W := \tH^0(\sco_\piii(1))$ of linear forms on 
$\piii$. Consider the Koszul complex$\, :$ 
\[
\cdots \lra \sco_\piii(-1) \otimes_k \overset{2}{\textstyle \bigwedge} W 
\overset{\displaystyle \delta_2}{\lra}  
\sco_\piii \otimes_k W \overset{\displaystyle \delta_1}{\lra} \sco_\piii(1) 
\lra 0 
\] 
corresponding the evaluation morphism $\delta_1$ and recall that 
$\Ker \delta_1 \simeq \Omega_\piii(1)$. One gets an epimorphism $\sco_\piii 
\otimes_k \overset{2}{\bigwedge} W \ra \Omega_\piii(2)$ which induces an 
isomorphism $\overset{2}{\bigwedge} W \izo \tH^0(\Omega_\piii(2))$. Recall, 
also, that $\delta_2(1) : \sco_\piii \otimes_k \overset{2}{\bigwedge} W \ra 
\sco_\piii(1) \otimes_k W$ maps $h \wedge h^\prime$ to $h \otimes h^\prime - 
h^\prime \otimes h$. 

\vskip2mm 

(ii) Let $L_1$, $L_2$ be disjoint lines in $\piii$. They correspond to a 
decomposition $W = W_1 \oplus W_2$, where $W_i = \tH^0(\sci_{L_i}(1))$, 
$i = 1,\, 2$. Consider the linear automorphism $\alpha_{12} := 
(- \text{id}_{W_1}) \oplus \text{id}_{W_2} : W \ra W$ and let $\theta_{12} : 
\Omega_\piii(1) \ra \sco_\piii(1)$ be the composite morphism$\, :$ 
\[
\Omega_\piii(1) \lra \sco_\piii \otimes_k W 
\xra{\displaystyle \alpha_{12}} \sco_\piii \otimes_k W 
\overset{\displaystyle \delta_1}{\lra} \sco_\piii(1)\, . 
\]   
If $h_0,\, h_1$ (resp., $h_2,\, h_3$) is a $k$-basis of $W_1$ (resp., $W_2$) 
then, by what has been said in (i)$\, :$  
\begin{gather*} 
\theta_{12}(1)(h_0 \wedge h_1) = 0\, ,\  \  \theta_{12}(1)(h_0 \wedge h_2) = 
2h_0h_2\, ,\  \  \theta_{12}(1)(h_0 \wedge h_3) = 2h_0h_3\\ 
\theta_{12}(1)(h_1 \wedge h_2) = 2h_1h_2\, ,\  \  \theta_{12}(1)(h_1 \wedge h_3) 
= 2h_1h_3\, ,\  \  \theta_{12}(1)(h_2 \wedge h_3) = 0\, .
\end{gather*} 
One deduces that $\text{Im}\, \theta_{12} = \sci_{L_1 \cup L_2}(1)$ and that one 
has an exact sequence$\, :$ 
\begin{equation}\label{E:kertheta} 
0 \lra 2\sco_\piii(-1) \lra \Omega_\piii(1) \xra{\displaystyle \theta_{12}} 
\sci_{L_1 \cup L_2}(1) \lra 0 
\end{equation}
where the left morphism is defined by the global sections $h_0 \wedge h_1$ and 
$h_2 \wedge h_3$ of $\Omega_\piii(2)$. 

\vskip2mm 

(iii) Consider a third line $L_3$ disjoint from $L_1$ and $L_2$. One can choose 
a $k$-basis $h_0,\, h_1$ (resp., $h_2,\, h_3$) of $W_1$ (resp., $W_2$) such 
that $h := h_0 + h_2$, $h^\prime := h_1 + h_3$ is a $k$-basis of 
$\tH^0(\sci_{L_3}(1))$. Then 
\[
\theta_{12}(1)(h \wedge h^\prime) = 2h_0h_3 - 2h_1h_2 
\]
is an equation of the unique quadric surface $Q$ containing $L_1 \cup L_2 \cup 
L_3$. The exact sequence \eqref{E:kertheta} induces an exact sequence$\, :$ 
\begin{equation}\label{E:3o(-1)omega(1)} 
0 \lra 3\sco_\piii(-1) \lra \Omega_\piii(1) \lra \sci_{L_1 \cup L_2,Q}(1) \lra 0 
\end{equation}
where the left morphism is defined by $h_0 \wedge h_1$, $h_2 \wedge h_3$ and 
$h \wedge h^\prime$. 
\end{remark}

\begin{remark}\label{R:en} 
We recall here, for the reader's convenience, the definition of a particular 
case of the \emph{Eagon-Northcott complex} and its basic property. 
Let $\phi : E \ra F$ be a morphism of vector bundles on an nonsingular 
quasi-projective variety $X$. Assume that $E$ has rank $n$ and $F$ has rank 
$n-1$. Then the Eagon-Northcott complex associated to $\phi$ is the 
complex$\, :$ 
\[
0 \lra F^\vee \otimes \overset{n}{\textstyle \bigwedge} E 
\overset{\displaystyle d_2}{\lra} \overset{n-1}{\textstyle \bigwedge} E 
\xra{\overset{n-1}{\bigwedge} {\displaystyle \phi}} 
\overset{n-1}{\textstyle \bigwedge} F 
\]
with $d_2$ defined by$\, :$ 
\[
d_2(f \otimes e_1 \wedge \ldots \wedge e_n) = 
{\textstyle \sum}_{i = 1}^n (-1)^{i-1} f(\phi(e_i)) e_1 \wedge \ldots \wedge 
{\widehat{e_i}} \wedge \cdots \wedge e_n\, . 
\]
The wedge product $E \times \overset{n-1}{\bigwedge} E \ra 
\overset{n}{\bigwedge} E$ induces an isomorphism$\, :$ 
\[
\overset{n-1}{\textstyle \bigwedge} E \simeq \sch om_{\sco_\piii}(E, 
\overset{n}{\textstyle \bigwedge} E) \simeq E^\vee \otimes 
\overset{n}{\textstyle \bigwedge} E 
\]
and, modulo this identification, $d_2$ can be identified with 
$\phi^\vee \otimes \text{id}_{\overset{n}{\bigwedge} E}$. 

It is a basic fact that if the degeneracy locus of $\phi$ has codimension 2 
in $X$ (or it is empty) then the Eagon-Northcott complex is exact.  
\end{remark}

\begin{lemma}\label{L:degeneracy} 
Under the hypothesis and with the notation from Lemma~\ref{L:qcapqprim}, 
consider, for $1 \leq i < j \leq 4$, the morphism $\theta_{ij} : \Omega_\piii(1) 
\ra \sco_\piii(1)$ with ${\fam0 Im}\, \theta_{ij} = \sci_{L_i \cup L_j}(1)$, defined 
in Remark~\ref{R:theta}(ii). Let us, also, denote by $Q_{ijl}$ the quadric 
surface containing $L_i \cup L_j \cup L_l$, $i < j < l$. Then$\, :$ 

\emph{(a)} The degeneracy scheme of$\, :$ 
\[
(\theta_{12},\, \theta_{34})^{\fam0 t} : \Omega_\piii(1) \lra 2\sco_\piii(1)  
\] 
is $L_1 \cup \ldots \cup L_4 \cup X$. 

\emph{(b)} The degeneracy scheme of$\, :$ 
\[
(\theta_{12},\, \theta_{23},\, \theta_{34})^{\fam0 t} : \Omega_\piii(1) \lra 
3\sco_\piii(1)  
\] 
is $Q_{123} \cup Q_{234}$.    
\end{lemma}

\begin{proof} 
(a) Let us denote the morphism from the statement by $\phi$. 

\vskip2mm 

\noindent
{\bf Claim.}\quad $\text{Supp}\, \Cok \phi \subseteq L_1 \cup \ldots \cup L_4 
\cup X$. 

\vskip2mm 

\noindent
\emph{Indeed}, applying the Snake Lemma to the diagram$\, :$ 
\[
\begin{CD}
0 @>>> 0 @>>> \Omega_\piii(1) @= \Omega_\piii(1) @>>> 0\\
@. @VVV @VV{\displaystyle \phi}V @VV{\displaystyle \theta_{12}}V\\
0 @>>> \sco_\piii(1) @>{\displaystyle \text{incl}_2}>> 2\sco_\piii(1) 
@>{\displaystyle \text{pr}_1}>> \sco_\piii(1) @>>> 0 
\end{CD}
\]
one gets an exact sequence$\, :$ 
\[
2\sco_\piii(-1) \overset{\displaystyle \partial}{\lra} \sco_\piii(1) \lra 
\Cok \phi \lra \sco_{L_1 \cup L_2}(1) \lra 0\, . 
\]
If $L_1$ (resp., $L_2$) has equations $h_0 = h_1 = 0$ (resp., $h_2 = h_3 = 0$) 
then $\partial$ is defined by $\theta_{34}(1)(h_0 \wedge h_1)$ and 
$\theta_{34}(1)(h_2 \wedge h_3)$. By Remark~\ref{R:theta}(iii), 
$\theta_{34}(1)(h_0 \wedge h_1)$ (resp., $\theta_{34}(1)(h_2 \wedge h_3)$) is an 
equation of $Q_{134}$ (resp., $Q_{234}$). Since, by Lemma~\ref{L:qcapqprim}, 
$Q_{134} \cap Q_{234} = L_3 \cup L_4 \cup X$ the claim follows. 

\vskip2mm 

Now, by Remark~\ref{R:en}, the Eagon-Northcott complex associated to $\phi$ is 
an exact sequence$\, :$ 
\begin{equation}\label{E:en} 
0 \lra 2\sco_\piii(-2) \xra{\displaystyle (\theta_{12}^\vee(-1), 
\theta_{34}^\vee(-1))} \Omega_\piii^2(2) \xra{\overset{2}{\bigwedge} 
{\displaystyle \phi}} \sco_\piii(2) 
\end{equation}
(recall that $\Omega_\piii^2(2) \simeq \text{T}_\piii(-2)$). One deduces that 
the map$\, :$ 
\[
\tH^0((\overset{2}{\textstyle \bigwedge} \phi)(1)) : \tH^0(\Omega_\piii^2(3)) 
\lra \tH^0(\sco_\piii(3)) 
\]
is injective hence its image has dimension 4. On the other hand, $\phi$ 
degenerates along $L_1 \cup \ldots \cup L_4$ hence 
$\text{Im}(\overset{2}{\bigwedge} \phi) \subseteq 
\sci_{L_1 \cup \ldots \cup L_4}(2)$. Since 
\[
\tH^0(\sci_{L_1 \cup \ldots \cup L_4}(3)) = 
\tH^0(\sci_{L_1 \cup \ldots \cup L_4 \cup X}(3)) 
\]
one deduces, from Lemma~\ref{L:il1l4x}, that the image of 
$(\overset{2}{\bigwedge} \phi)(1)$ is $\sci_{L_1 \cup \ldots \cup L_4 \cup X}(3)$. 

\vskip2mm 

(b) Let us denote by $\psi$ the morphism from the statement.  

\vskip2mm 

\noindent
{\bf Claim.}\quad $\theta_{12},\, \theta_{23},\, \theta_{34} \in 
\text{Hom}_{\sco_\piii}(\Omega_\piii(1), \sco_\piii(1))$ \emph{are linearly 
independent}.  

\vskip2mm 

\noindent
\emph{Indeed}, take a fifth line $L_5$, disjoint form each of the lines 
$L_1, \ldots , L_4$ and such that $L_1 \cup \ldots \cup L_5$ admits no 
5-secant. Let $h = h^\prime = 0$ be equations of $L_5$. Then, according to 
Remark~\ref{R:theta}(iii), $\theta_{ij}(1)(h \wedge h^\prime)$ is an equation 
of $Q_{ij5}$, for $1 \leq i < j \leq 4$. Since $Q_{345}$ does not contain $L_2$, 
the equation of $Q_{345}$ cannot be a linear combination of the equations of 
$Q_{125}$ and $Q_{235}$. It follows that $\theta_{12}(1)(h \wedge h^\prime)$, 
$\theta_{23}(1)(h \wedge h^\prime)$ and $\theta_{34}(1)(h \wedge h^\prime)$ are 
linearly independent whence the claim. 

\vskip2mm 

Now, using the exact sequence \eqref{E:en}, one deduces that 
\[
(\theta_{12}^\vee,\, \theta_{23}^\vee,\, \theta_{34}^\vee) : 3\sco_\piii(-1) \lra 
\text{T}_\piii(-1) \simeq \Omega_\piii^2(3)  
\] 
does not degenerate on the whole $\piii$ hence the same is true for $\psi$. 
It follows that the degeneracy scheme of $\psi$ is a surface $\Sigma$ of 
degree 4 in $\piii$. But the exact sequences \eqref{E:kertheta} and 
\eqref{E:3o(-1)omega(1)} from Remark~\ref{R:theta} show that 
$(\theta_{12},\, \theta_{23})^{\text{t}} : \Omega_\piii(1) \ra 2\sco_\piii(1)$ 
degenerates along $Q_{123}$. Analogously, $(\theta_{23},\, \theta_{34})^{\text{t}} :
\Omega_\piii(1) \ra 2\sco_\piii(1)$ degenerates along $Q_{234}$. It follows that 
$\Sigma = Q_{123} \cup Q_{234}$.  
\end{proof}

\begin{lemma}\label{L:lprimcapz} 
Let $Z$ be a locally Cohen-Macaulay curve in $\piii$, supported on a line $L$, 
and let $Q \supset L$ be a nonsigular quadric surface. If, for every line 
$L^\prime$ from the other ruling of $Q$, one has $L^\prime \cap Z = $ one simple 
point then $Z = L$ as schemes. 
\end{lemma}

\begin{proof} 
We have to show that $\text{deg}\, Z = 1$. Suppose that $\text{deg}\, Z \geq 
2$.  Then $Z$ contains, as a subscheme, a double structure $X$ on L. 
One can assume that $L$ is the line of equations $x_2 = x_3 = 0$ and that $Q$ 
is the quadric of equation $x_0x_3 - x_1x_2 = 0$. In this case, it is well 
known that $I(X) = (ax_3 - bx_2,\, x_2^2,\, x_2x_3,\, x_3^2)$, with $a,\, b \in 
k[x_0,x_1]$ coprime homogeneous polynomials of the same degree. 

Now, any line $L^\prime$ from the other ruling of $Q$ can be represented 
parametrically as the image of a morphism$\, :$ 
\[
\pj \lra \piii\, ,\  (u:t) \mapsto (c_0u:c_1u:c_0t:c_1t)  
\] 
for some $(c_0,c_1) \in k^2 \setminus \{(0,0)\}$. Choose $c_0,\, c_1$ such 
that $a(c_0,c_1)c_1 - b(c_0,c_1)c_0 = 0$. Then $ax_3 - bx_2 \vb L^\prime = 0$ 
hence $L^\prime \cap X$ is a double point on $L^\prime$ which \emph{contradicts} 
our hypothesis. 
\end{proof}

\begin{remark}\label{R:chern} 
Let $E$ be a rank 3 vector bundle on $\piii$, $Z$ a closed subscheme of 
$\piii$ of dimension $\leq 1$ and $m$ an integer. Assume that one has an 
epimorphism $\sigma : E \ra \sci_Z(m)$. In this case $\scf := \Ker \sigma$ is 
a rank 2 reflexive sheaf and one has an exact sequence$\, :$ 
\[
0 \lra \scf \lra E \lra \sci_Z(m) \lra 0\, .
\]
Let $Z_{\text{CM}}$ be the largest closed subscheme of $Z$ which is locally 
Cohen-Macaulay of pure dimension 1 (or empty if $\dim Z \leq 0$). One has an 
exact sequence$\, :$ 
\[
0 \lra \sct \lra \sco_Z \lra \sco_{Z_{\text{CM}}} \lra 0 
\]
with $\dim \text{Supp}\, \sct \leq 0$. The Hilbert polynomial of $\sco_Z$ has 
the form $\chi(\sco_Z(t)) = dt + \chi(\sco_Z)$, for some nonnegative integer 
$d$ which we denote by $\text{deg}\, Z$ (such that $\text{deg}\, Z = 0$ if 
$\dim Z \leq 0$). One has $\text{deg}\, Z = \text{deg}\, Z_{\text{CM}}$ and 
$\chi(\sco_Z) = \chi(\sco_{Z_{\text{CM}}}) + \text{length}\, \sct$. 

\vskip2mm 

We assert that, under the above hypotheses, one has $c_3(\scf) = 
\text{length}\, \sct$ and$\, :$ 
\begin{gather*} 
c_1(E) = c_1(\scf) + m\, ,\  \  c_2(E) = c_2(\scf) + mc_1(\scf) + 
\text{deg}\, Z_{\text{CM}}\, ,\\
c_3(E) = - c_3(\scf) + mc_2(\scf) + (c_1(\scf) - m + 4)\text{deg}\, Z_{\text{CM}} 
- 2\chi(\sco_{Z_{\text{CM}}})\, .  
\end{gather*}
\emph{Indeed}, it follows from \cite[Prop.~2.6]{ha} that $c_3(\scf) = 
\text{length}\, \sce xt^1(\scf, \omega_\piii)$. But 
\[
\sce xt^1(\scf, \omega_\piii) \simeq \sce xt^2(\sci_Z(m), \omega_\piii) 
\simeq \sce xt^3(\sco_Z(m), \omega_\piii) \simeq \sce xt^3(\sct(m), \omega_\piii) 
\] 
and $\text{length}\, \sce xt^3(\sct(m), \omega_\piii) = \text{length}\, \sct$. 

On the other hand, one can easily prove, using Riemann-Roch, that 
\[
c_1(\sci_Z(m)) = m\, ,\  \  c_2(\sci_Z(m)) = \text{deg}\, Z\, ,\  \  
c_3(\sci_Z(m)) = (4-m)\text{deg}\, Z - 2\chi(\sco_Z)
\] 
and our assertion follows. 
\end{remark}

We are now ready to give the 

\begin{proof}[Proof of Lemma~\ref{L:omega(1)iy(3)}] 
We begin by fixing some notation. 
Let us denote the union $L_1 \cup \ldots \cup L_5$ by $Y$. For $1 \leq i < j 
< l \leq 5$, let $Q_{ijl}$ be the quadric surface containing $L_i \cup L_j \cup 
L_l$ and let $q_{ijl} \in \tH^0(\sco_\piii(2))$ be an equation of $Q_{ijl}$. 
Consider, also, for $1 \leq i < j \leq 4$ the morphism 
\[
\theta_{ij} : \Omega_\piii(1) \lra \sco_\piii(1) 
\] 
with image $\sci_{L_i \cup L_j}(1)$ defined in Remark~\ref{R:theta}(ii). We 
denote by $t_{ij}$ the global section of $\text{T}_\piii$ corresponding to 
\[
\theta_{ij}^\vee(1) : \sco_\piii \lra \text{T}_\piii \, . 
\]

Now, we will show that, for general constants $a_1,\, a_2,\, a_3 \in k$, the 
image of the morphism 
\[
\sigma := a_1q_{345}\theta_{12} + a_2q_{145}\theta_{23} + a_3q_{125}\theta_{34} : 
\Omega_\piii(1) \lra \sco_\piii(3) 
\]
is $\sci_Y(3)$. 

\vskip2mm 

\noindent 
\emph{Indeed}, $\sigma^\vee(3) : \sco_\piii \ra \text{T}_\piii(2)$ is defined 
by the global section 
\[
s := a_1q_{345}t_{12} + a_2q_{145}t_{23} + a_3q_{125}t_{34}
\] 
of $\text{T}_\piii(2)$. If $Z$ is the zero scheme of $s$ then the image of 
$\sigma$ is $\sci_Z(3)$. Since, by Lemma~\ref{L:degeneracy}(b), the 
dependence locus of $t_{12},\, t_{23},\, t_{34}$ is $Q_{123} \cup Q_{234}$ and 
since, by Lemma~\ref{L:l5} applied to the lines $L_3,\, L_4,\, L_1,\, L_2,\, 
L_5$, one has 
\[
Q_{345} \cap Q_{145} \cap Q_{125} = L_5 \cup \Gamma_4 \cup \Gamma_1\, ,  
\] 
with $\Gamma_4 \subset L_4$ and $\Gamma_1 \subset L_1$, 
it follows that if $a_i \neq 0$, $i = 1,\, 2,\, 3$, then $Z \subseteq Q_{123} 
\cup Q_{234} \cup L_5$ as sets. 

We formulate, now, three claims. In each of them one assumes that the 
constants $a_1,\, a_2,\, a_3$ are general.  

\vskip2mm 

\noindent
{\bf Claim 1.}\quad \emph{If} $L \subset Q_{123}$ \emph{is a line intersecting} 
$L_1,\, L_2,\, L_3$ \emph{then} $L \cap Z = L \cap Y$ \emph{as schemes}. 

\vskip2mm 

\noindent
{\bf Claim 2.}\quad \emph{If} $L \subset Q_{234}$ \emph{is a line intersecting} 
$L_2,\, L_3,\, L_4$ \emph{then} $L \cap Z = L \cap Y$ \emph{as schemes}. 

\vskip2mm 

\noindent
{\bf Claim 3.}\quad \emph{If} $L \subset Q_{125}$ \emph{is a line intersecting} 
$L_1,\, L_2,\, L_5$ \emph{then} $L \cap Z = L \cap Y$ \emph{as schemes}. 

\vskip2mm 

Assuming the claims, for the moment, one deduces, from 
Lemma~\ref{L:lprimcapz}, that $Z_{\text{CM}} = Y$ as schemes (see 
Remark~\ref{R:chern} for the notation). Applying the Chern classes formulae 
from Remark~\ref{R:chern} to the exact sequence$\, :$ 
\[
0 \lra \scf \lra \Omega_\piii(1) \overset{\displaystyle \sigma}{\lra} 
\sci_Z(3) \lra 0 
\]    
(with $\scf := \Ker \sigma$) one gets that $c_3(\scf) = 0$ hence $Z = 
Z_{\text{CM}}$ as schemes and Lemma~\ref{L:omega(1)iy(3)} is proven. 

\vskip2mm 

Let us, finally, prove the three claims. We recall that, for every line 
$L \subset \piii$, one has $\text{T}_\piii \vb L \simeq \sco_L(2) \oplus 
2\sco_L(1)$. The zero scheme $Z(t_{ij} \vb L)$ of the global section 
$t_{ij} \vb L$ of $\text{T}_\piii \vb L$ is $L \cap (L_i \cup L_j)$. 

\vskip2mm 

\noindent
\emph{Proof of Claim 1.}\quad We denote by $P_i$ the intersection point of 
$L$ and $L_i$, $i = 1,\, 2,\, 3$. In case $L$ intersects $L_4$ (resp., $L_5$) 
we denote by $P_4$ (resp., $P_5$) their intersection point. One has to 
consider three cases.  

\vskip2mm     

\noindent 
$\bullet$\quad If $L \cap L_4 = \emptyset$ and $L \cap L_5 = \emptyset$ then, 
as divisors on $L$, 
\begin{gather*} 
Z(t_{12} \vb L) = P_1 + P_2\, ,\  \  Z(t_{23} \vb L) = P_2 + P_3\, ,\  \  
Z(t_{34} \vb L) = P_3\, ,\\ 
L \cap Q_{345} = P_3 + P_3^\prim \, ,\  \  L \cap Q_{145} = P_1 + P_1^\prim 
\, ,\  \  L \cap Q_{125} = P_1 + P_2  
\end{gather*}
for some points $P_1^\prim ,\, P_3^\prim$ of $L$. It follows that, choosing a 
convenient isomorphism $\text{T}_\piii \vb L \simeq \sco_L(2) \oplus 
2\sco_L(1)$, one has 
\[
t_{12} \vb L = (f_{12},0,0)\, ,\  \  t_{23} \vb L = (f_{23},0,0)\, ,\  \  
t_{34} \vb L = (f_{34}, \ell ,0)
\]
with $0 \neq f_{12}$ vanishing at $P_1$ and $P_2$, with $0 \neq f_{23}$ 
vanishing at $P_2$ and $P_3$, and with $f_{34}$ and $0 \neq \ell$ vanishing at 
$P_3$. One deduces easily that if $a_3 \neq 0$ then $L \cap Z = P_1 + P_2 + 
P_3$ as divisors on $L$.  

\vskip2mm 

\noindent
$\bullet$\quad If $L \cap L_4 \neq \emptyset$ then $L \cap L_5 = \emptyset$ 
and 
\begin{gather*}
Z(t_{12} \vb L) = P_1 + P_2\, ,\  \  Z(t_{23} \vb L) = P_2 + P_3\, ,\  \  
Z(t_{34} \vb L) = P_3 + P_4\, ,\\ 
L \cap Q_{345} = P_3 + P_4 \, ,\  \  L \cap Q_{145} = P_1 + P_4 
\, ,\  \  L \cap Q_{125} = P_1 + P_2\, .   
\end{gather*} 
It follows that, for general $a_1,\, a_2,\, a_3 \in k$, $L \cap Z = P_1 + P_2 
+ P_3 + P_4$ as divisors on $L$. Notice that there are at most two lines 
intersecting each of the lines $L_1,\, L_2,\, L_3,\, L_4$. 

\vskip2mm 

\noindent
$\bullet$\quad If $L \cap L_5 \neq \emptyset$ then $L \cap L_4 = \emptyset$ 
and 
\begin{gather*}
Z(t_{12} \vb L) = P_1 + P_2\, ,\  \  Z(t_{23} \vb L) = P_2 + P_3\, ,\  \  
Z(t_{34} \vb L) = P_3\, ,\\ 
L \cap Q_{345} = P_3 + P_5 \, ,\  \  L \cap Q_{145} = P_1 + P_5 
\, ,\  \  Q_{125} \supset L \, .  
\end{gather*} 
It follows that if $a_1,\, a_2 \in k$ are general and if $a_3 \in k$ is 
arbitrary then $L \cap Z = P_1 + P_2 + P_3 + P_5$ as divisors on $L$. Notice 
that are at most two lines intersecting each of the lines $L_1,\, L_2,\, 
L_3,\, L_5$. 

\vskip2mm 

\noindent 
\emph{Proof of Claim 2.}\quad 
One has to consider three cases. 

\vskip2mm 

\noindent
$\bullet$\quad If $L \cap L_1 = \emptyset$ and $L \cap L_5 = \emptyset$ then, 
as divisors on $L$, 
\begin{gather*}
Z(t_{12} \vb L) = P_2\, ,\  \  Z(t_{23} \vb L) = P_2 + P_3\, ,\  \  
Z(t_{34} \vb L) = P_3 + P_4\, ,\\ 
L \cap Q_{345} = P_3 + P_4\, ,\  \  L \cap Q_{145} = P_4 + P_4^\prim \, ,\  \  
L \cap Q_{125} = P_2 + P_2^\prim  
\end{gather*}
for some points $P_2^\prim ,\, P_4^\prim$ of $L$. One deduces, as in the first 
case of the proof of Claim 1, that if $a_1 \neq 0$ then $L \cap Z = P_2 + P_3 
+ P_4$ as divisors on $L$. 

\vskip2mm 

\noindent
$\bullet$\quad If $L \cap L_1 \neq \emptyset$ then $L \cap L_5 = \emptyset$ 
and 
\begin{gather*}
Z(t_{12} \vb L) = P_1 + P_2\, ,\  \  Z(t_{23} \vb L) = P_2 + P_3\, ,\  \  
Z(t_{34} \vb L) = P_3 + P_4\, ,\\ 
L \cap Q_{345} = P_3 + P_4\, ,\  \  L \cap Q_{145} = P_1 + P_4 \, ,\  \  
L \cap Q_{125} = P_1 + P_2\, . 
\end{gather*}    
It follows that, for general $a_1,\, a_2,\, a_3 \in k$, $L \cap Z = P_1 + P_2 + 
P_3 + P_4$ as divisors on $L$. Notice that there are at most two lines 
intersecting each of the lines $L_1,\, L_2,\, L_3,\, L_4$. 

\vskip2mm 

\noindent
$\bullet$\quad If $L \cap L_5 \neq \emptyset$ then $L \cap L_1 = \emptyset$ 
and 
\begin{gather*}
Z(t_{12} \vb L) = P_2\, ,\  \  Z(t_{23} \vb L) = P_2 + P_3\, ,\  \  
Z(t_{34} \vb L) = P_3 + P_4\, ,\\ 
Q_{345} \supset L\, ,\  \  L \cap Q_{145} = P_4 + P_5 \, ,\  \  
L \cap Q_{125} = P_2 + P_5\, . 
\end{gather*}    
It follows that if $a_2,\, a_3 \in k$ are general and $a_1 \in k$ is arbitrary 
then $L \cap Z = P_2 + P_3 + P_4 + P_5$ as divisors on $L$.  Notice that there 
are at most two lines intersecting each of the lines $L_2,\, L_3,\, L_4,\, 
L_5$.  

\vskip2mm

\noindent
\emph{Proof of Claim 3.}\quad One has to consider three cases. 

\vskip2mm 

\noindent
$\bullet$\quad If $L \cap L_3 = \emptyset$ and $L \cap L_4 = \emptyset$ then, 
as divisors on $L$, 
\begin{gather*}
Z(t_{12} \vb L) = P_1 + P_2\, ,\  \  Z(t_{23} \vb L) = P_2\, ,\  \  
Z(t_{34} \vb L) = \emptyset \, ,\\ 
L \cap Q_{345} = P_5 + P_5^\prim \, ,\  \  L \cap Q_{145} = P_1 + P_5\, ,\  \  
Q_{125} \supset L 
\end{gather*} 
for some point $P_5^\prim$ of $L$. One deduces, as in the first case of the 
proof of Claim 1, that if $a_2 \neq 0$ then $L \cap Z = P_1 + P_2 + P_5$ as 
divisors on $L$. 

\vskip2mm 

\noindent
$\bullet$\quad If $L \cap L_3 \neq \emptyset$ then $L \cap L_4 = \emptyset$ 
and 
\begin{gather*}
Z(t_{12} \vb L) = P_1 + P_2\, ,\  \  Z(t_{23} \vb L) = P_2 + P_3\, ,\  \  
Z(t_{34} \vb L) = P_3\, ,\\ 
L \cap Q_{345} = P_3 + P_5 \, ,\  \  L \cap Q_{145} = P_1 + P_5\, ,\  \  
Q_{125} \supset L\, . 
\end{gather*}
It follows that if $a_1,\, a_2 \in k$ are general and $a_3 \in k$ is arbitrary 
then $L \cap Z = P_1 + P_2 + P_3 + P_5$ as divisors on $L$. Notice that there 
are at most two lines intersecting each of the lines $L_1,\, L_2,\, L_3,\, 
L_5$. 

\vskip2mm 

\noindent
$\bullet$\quad If $L \cap L_4 \neq \emptyset$ then $L \cap L_3 = \emptyset$ 
and 
\begin{gather*}
Z(t_{12} \vb L) = P_1 + P_2\, ,\  \  Z(t_{23} \vb L) = P_2\, ,\  \  
Z(t_{34} \vb L) = P_4\, ,\\ 
L \cap Q_{345} = P_4 + P_5 \, ,\  \  Q_{145} \supset L\, ,\  \  
Q_{125} \supset L\, .  
\end{gather*}  
It follows that if $a_1 \neq 0$ then $L \cap Z = P_1 + P_2 + P_4 + P_5$ as 
divisors on $L$. 

This concludes the proof of the three claims above and, consequently, of 
Lemma~\ref{L:omega(1)iy(3)}.  
\end{proof}

\begin{remark}\label{R:questions} 
We do not know how to characterize the 4-instantons $F$ with $F(2)$ globally 
generated. Actually, we cannot answer even a simpler question$\, :$ 
as we recalled at the beginning of the paper, if $F$ is a 4-instanton then 
$F(4)$ is globally generated. Is it true that if $F$ has no jumping line of 
maximal order 4 then $F(3)$ is globally generated ?  
\end{remark}

\newpage 


\noindent 
{\bf Acknowledgements.}\quad The special form of the morphisms $\sigma : 
\Omega_\piii(1) \ra \sco_\piii(3)$ used in the proof of 
Lemma~\ref{L:omega(1)iy(3)} was ``guessed'' after several experiments using 
the progam \texttt{Macaulay2} of Grayson and Stillman \cite{mac}. 
N. Manolache expresses his thanks to Udo Vetter and the Institute of Mathematics, 
Oldenburg University, for warm hospitality during the preparation of this 
paper.

\end{document}